\documentclass[12pt]{article}

\usepackage{dsfont}
\usepackage{amsfonts,amsmath,amsthm}
\usepackage{algorithm, algorithmic}
\usepackage{color}
\usepackage{graphicx}
\usepackage{stmaryrd}        % provides \llbracket and \rrbracket

\usepackage[paper=a4paper,dvips,top=2cm,left=2cm,right=2cm,
    foot=1cm,bottom=4cm]{geometry}

\begin{document}

\title{A Note on Quaternion Skew-Symmetric Matrices}
\author{ Liqun Qi\footnote{%
    Department of Applied Mathematics, The Hong Kong Polytechnic University, Hung Hom,
    Kowloon, Hong Kong
    %Department of Mathematics, School of Science, Hangzhou Dianzi University, Hangzhou 310018 China
    ({\tt maqilq@polyu.edu.hk}).}
    \and and \
    Ziyan Luo\footnote{Department of Mathematics,
  Beijing Jiaotong University, Beijing 100044, China. (zyluo@bjtu.edu.cn). This author's work was supported by NSFC (Grant No.  11771038) and Beijing Natural Science Foundation (Grant No.  Z190002).}
}
\date{\today}
\maketitle

\begin{abstract}
The product of a complex skew-symmetric matrix and its conjugate transpose is a positive  semi-definite Hermitian matrix with nonnegative eigenvalues, with a property that each distinct positive eigenvalue has even multiplicity.   This property plays a key role for Professor Loo-Keng Hua to establish the unitary equivalence theorem for complex skew-symmetric matrices.   We show that this property is no longer true for quaternion skew-symmetric matrices.   This poses a difficulty for finding the canonical form of a quaternion skew-symmetric matrix under unitary equivalence, which may be crucial for the spectral theory of dual quaternion matrices.   We show that the inverse of a quaternion skew-symmetric matrix may not be skew-symmetric though the inverse of a nonsingular complex skew-symmetric matrix is always skew-symmetric.   We also present the concept of basic quaternion skew-symmetric matrices.

\medskip

  \medskip

  \textbf{Key words.} skew-symmetric matrices, complex matrices, quaternion matrices, unitary equivalence.

  %\medskip
  %\textbf{AMS subject classifications.}
\end{abstract}

\renewcommand{\Re}{\mathds{R}}
\newcommand{\rank}{\mathrm{rank}}
\renewcommand{\span}{\mathrm{span}}
\newcommand{\X}{\mathcal{X}}
\newcommand{\A}{\mathcal{A}}
\newcommand{\I}{\mathcal{I}}
\newcommand{\B}{\mathcal{B}}
\newcommand{\C}{\mathcal{C}}
\newcommand{\OO}{\mathcal{O}}
\newcommand{\e}{\mathbf{e}}
\newcommand{\0}{\mathbf{0}}
\newcommand{\dd}{\mathbf{d}}
\newcommand{\ii}{\mathbf{i}}
\newcommand{\jj}{\mathbf{j}}
\newcommand{\kk}{\mathbf{k}}
\newcommand{\va}{\mathbf{a}}
\newcommand{\vb}{\mathbf{b}}
\newcommand{\vc}{\mathbf{c}}
\newcommand{\vg}{\mathbf{g}}
\newcommand{\vr}{\mathbf{r}}
\newcommand{\vt}{\rm{vec}}
\newcommand{\vx}{\mathbf{x}}
\newcommand{\vy}{\mathbf{y}}
\newcommand{\vu}{\mathbf{u}}
\newcommand{\vv}{\mathbf{v}}
\newcommand{\y}{\mathbf{y}}
\newcommand{\vz}{\mathbf{z}}
\newcommand{\T}{\top}

\newtheorem{Thm}{Theorem}[section]
\newtheorem{Def}[Thm]{Definition}
\newtheorem{Ass}[Thm]{Assumption}
\newtheorem{Lem}[Thm]{Lemma}
\newtheorem{Prop}[Thm]{Proposition}
\newtheorem{Cor}[Thm]{Corollary}
\newtheorem{example}[Thm]{Example}

\section{Introduction}

In 1944, Loo-Keng Hua established the following theorem for unitary equivalence of complex skew-symmetric matrices.   See Theorem 7 of \cite{Hu44}.

\begin{Thm} \label{t1.1}
Suppose that $Z$ is an $n \times n$ nonsingular complex skew-symmetric matrix.   Then we have an $n \times n$  unitary matrix $U$ such that
\begin{equation}
UZU^\top = \Sigma,
\end{equation}
where $\Sigma$ is a block diagonal matrix,
$$\Sigma = {\rm diag}\left(\begin{bmatrix} 0 & \sigma_1 \\ -\sigma_1 & 0 \end{bmatrix}, \cdots,
\begin{bmatrix} 0 & \sigma_k \\ -\sigma_k & 0 \end{bmatrix}\right),$$
$\sigma_1, \cdots, \sigma_k$ are real nonzero numbers, $\sigma_1^2, \sigma_1^2, \cdots, \sigma_k^2, \sigma_k^2$ are the eigenvalues of the positive definite Hermitian matrix $-Z\bar Z$,   and $n=2k$.
\end{Thm}

Hua \cite{Hu44} pointed out that this theorem may be extended to the singular case without any difficulty.

The proof of Hua's theorem is elegant.   Unaware of this result, in 1960, Stander and Wiegmann \cite{SW60} reproved this result and extended it to quaternion $*-$skew-symmetric matrices.   In 1961, Youla \cite{Yo61} pointed out this,  and established a unitary equivalence result for more general complex matrices.

In \cite{QL21}, we introduced eigenvalues and subeigenvalues for dual complex matrices.   We derived an eigenvalue and subeigenvalue decomposition theorem for dual complex Hermitian matrices.   Then we presented singular value decomposition for general dual complex matrices.   In that process, we applied Theorem \ref{t1.1}.

The product of a complex matrix and its conjugate transpose is a positive semi-definite Hermitian matrix with nonnegative eigenvalues.    The product of a quaternion matrix and its conjugate transpose is still a positive semi-definite Hermitian matrix with nonnegative right eigenvalues.

For a complex skew-symmetric matrix, each distinct positive eigenvalue of such a product has even multiplicity.   The proof of Theorem \ref{t1.1} used this property of complex skew matrices.   In this paper, we show that this property does not hold for quaternion skew-symmetric matrices.   This poses a difficulty for finding the canonical form of a quaternion skew-symmetric matrix under unitary equivalence, which may be crucial for the spectral theory of dual quaternion matrices.

In the next section, we summarize some preliminary knowledge on quaternions and quaternion matrices.

We show in Section 3 that if one tries to extend Theorem \ref{t1.1} to quaternion matrices directly, then it is necessary to show that each distinct positive right eigenvalue of the product of a quaternion skew-symmetric matrix and its conjugate transpose has even multiplicity.

In Section 4, we present a throughout analysis on this issue for $2\times 2$ and $3\times 3$ quaternion skew-symmetric matrices.   We show that the product of a $2\times 2$ nonzero quaternion skew-symmetric matrix and its conjugate transpose has only one double positive right eigenvalue $1$.   Thus, this property is not violated in the $2\times 2$ case.    Then we present a condition for a $3\times 3$ nonzero quaternion skew-symmetric matrix that its product with its conjugate transpose has a zero right eigenvalue and a double positive right eigenvalue.   If this condition is violated, then that product has three positive right eigenvalues.   We show that these positive right eigenvalues can be distinct to each other.
Hence, Theorem \ref{t1.1} cannot be extended to quaternion matrices directly.    Some changes are needed.

It is a common knowledge that the inverse of a nonsingular complex skew-symmetric matrix is always skew-symmetric.   In Section 5, we show that the inverse of a $3\times 3$ quaternion skew-symmetric matrix, if exists, is not skew-symmetric.

Then, in Section 6, we present the concept of basic quaternion skew-symmetric matrices.

In Section 7, we show that the canonical form of a quaternion skew-symmetric matrix under unitary equivalence is important for studying spectral theory of dual quaternion matrices, while dual quaternions are widely used in robotics and computer graphics.  Thus, it is a challenging and meaningful task to extend Theorem \ref{t1.1} to quaternion skew-symmetric matrices.

\section{Quaternions and Quaternion Matrices}

In this section, we collect some basic knowledge about quaternions and quaternion matrices \cite{Ch98, Ro14, WLZZ18, Zh97}.

\subsection{Quaternions}

Let ${\mathbb R}$, $\mathbb C$ and $\mathbb Q$ be the sets of the real numbers, the complex numbers and the quaternions, respectively.  Denote scalars, vectors and matrices by small letters, bold small letters and capital letters, respectively.
A quaternion $q$ has the form
$$q = q_0 + q_1\ii + q_2\jj + q_3\kk,$$
where $q_0, q_1, q_2$ and $q_3$ are real numbers, $\ii, \jj$ and $\kk$ are three imaginary units of quaternions, satisfying
$$\ii^2 = \jj^2 = \kk^2 =\ii\jj\kk = -1,$$
$$\ii\jj = -\jj\ii = \kk, \ \jj\kk = - \kk\jj = \ii, \kk\ii = -\ii\kk = \jj.$$
The real part of $q$ is Re$(q) = q_0$.   The imaginary part of $q$ is Im$(q) = q_1\ii + q_2\jj +q_3\kk$.
A quaternion is called imaginary if its real part is zero.
The multiplication of quaternions satisfies the distribution law, but is noncommutative.

The conjugate of $q = q_0 + q_1\ii + q_2\jj + q_3\kk$ is
$$\bar q \equiv q^* = q_0 - q_1\ii - q_2\jj - q_3\kk.$$
The magnitude of $q$ is
$$|q| = \sqrt{q_0^2+q_1^2+q_2^2+q_3^2}.$$

It follows the inverse of a nonzero quaternion $q$ is given by
$$q^{-1} = {\bar q \over |q|^2}.$$

Two quaternions $p$ and $q$ are said to be similar if there exists a nonzero $u$ such that $u^{-1}pu = q$.   This is denoted as $p \sim q$.  Then $\sim$ is an equivalence relation.  Denote the equivalence class containing $q$ by $[q]$.    If $p \sim q$, then $|p|=|q|$.    Note that if $q$ is real, $[q]$ is a singleton.

A quaternion $q$ is called a unit quaternion if $|q|=1$.    Let $S^3$ be the set of all unit quaternions, and $S^2$ be the set of all unit imaginary quaternions.   Then, by \cite{Ch98}, for any $q \in S^3$, the Euler formula
\begin{equation} \label{ee1}
q = e^{\omega \theta} = \cos \theta + \omega \sin \theta,
\end{equation}
still holds, where $\omega \in S^2$ and $\theta \in {\mathbb R}$.

Denote the $n$-dimensional quaternion vector space by ${\mathbb Q}^n$.  For
$\vx = (x_1, \cdots, x_n)^\top \in {\mathbb Q}^n$, we have
$\vx^* = (\bar x_1, \cdots, \bar x_n)$, and
$$\|\vx\|^2 = \vx^*\vx = |x_1|^2 + \cdots, |x_n|^2.$$

Let $\vx, \vy \in {\mathbb Q}^n$. If $\vx^*\vy = 0$, then we say that $\vx$ and $\vy$ are orthogonal.
If $\vx^{(1)}, \cdots, \vx^{(n)} \in {\mathbb Q}^n$,
$$\left(\vx^{(i)}\right)^*\vx^{(j)} = \delta_{ij},$$
for $i, j = 1, \cdots, n$, where $\delta_{ij}$ is the Kronecker symbol, then $\left\{ \vx^{(1)}, \cdots, \vx^{(n)} \right\}$ forms an orthonormal basis of ${\mathbb Q}^n$.

\subsection{Quaternion Matrices}

Denote the collections of real, complex and quaternion $m \times n$ matrices by ${\mathbb R}^{m \times n}$, ${\mathbb C}^{m \times n}$ and ${\mathbb Q}^{m \times n}$, respectively.

A quaternion matrix $A= (a_{ij}) \in {\mathbb Q}^{m \times n}$ can be denoted as
\begin{equation} \label{ee2}
A = A_0 + A_1\ii + A_2\jj + A_3\kk,
\end{equation}
where $A_0, A_1, A_2, A_3 \in {\mathbb R}^{m \times n}$.   The transpose of $A$ is $A^\top = (a_{ji})$. The conjugate of $A$ is $\bar A = (\bar a_{ij})$.   The conjugate transpose of $A$ is $A^* = (\bar a_{ji}) = \bar A^\top$.

%For $A, B \in {\mathbb DC}^{m \times n}$, their inner product is defined as
%$$\langle A, B \rangle = {\rm Tr}(A^*B),$$
%where ${\rm Tr}(A^*B)$ denotes the trace of $A^*B$.
The Frobenius norm of $A$ is
$$\|A\|_F  = \sqrt{\sum_{i=1}^m \sum_{j=1}^n |a_{ij}|^2}.$$
%$$\|A\|_F = \sqrt{\langle A, A \rangle} = \sqrt{{\rm Tr}(A^*A)} = \sqrt{\sum_{i=1}^m \sum_{j=1}^n |a_{ij}|^2}.$$
%the $\ell_1$-norm of $A = (a_{ij}) \in {\mathbb Q}^{m \times n}$ is defined by $\|A\|_1 = \sum_{i=1}^m \sum_{j=1}^n |a_{ij}|$, and the $\ell_\infty$-norm of $A$ is defined by $\|A\|_\infty = \max_{i, j} |a_{ij}|$ \cite{JNS19}.

Let $A \in {\mathbb Q}^{m \times n}$ and $B \in {\mathbb Q}^{n \times r}$.   Then we have $(AB)^* = B^*A^*$.   But in general, $(AB)^\top \not = B^\top A^\top$ and $\overline {AB} \not = \bar A \bar B$ in general.

A square quaternion matrix $A \in {\mathbb Q}^{n \times n}$ is called normal if $A^*A = AA^*$, Hermitian if $A^* = A$; unitary if $A^*A = I$; and invertible (nonsingular) if $AB = BA = I$ for some $B \in {\mathbb Q}^{n \times n}$.
We have $(AB)^{-1} = B^{-1}A^{-1}$ if $A$ and $B$ are invertible, and $\left(A^*\right)^{-1} = \left(A^{-1}\right)^*$ if $A$ is invertible.

A Hermitian matrix $A \in {\mathbb Q}^{n \times n}$ is called positive semi-definite if for any $\vx \in {\mathbb Q}$, $\vx^*A\vx \ge 0$; $A$ is called positive definite if for any $\vx \in {\mathbb Q}$ with $\vx \not = \0$,  we have $\vx^*A\vx > 0$.

A square quaternion matrix $A \in {\mathbb Q}^{n \times n}$ is unitary if and only if its column (row) vectors form an orthonormal basis of ${\mathbb Q}^n$.

Suppose that $A \in {\mathbb Q}^{n \times n}$, $\vx \in {\mathbb Q}^n$, $\vx \not = \0$, and $\lambda \in
{\mathbb Q}$, satisfy
\begin{equation} \label{ee3}
A\vx = \vx\lambda.
\end{equation}
Then $\lambda$ is called a right eigenvalue of $A$, with $\vx$ as an associated right eigenvector.  If $\lambda$ is complex and its imaginary part is nonnegative, then $\lambda$ is called a standard right eigenvalue of $A$.    The following theorem collects some know results of the right eigenvalues of quaternion matrices \cite{WLZZ18}.

\begin{Thm} \label{t2.1}
Suppose that $A \in {\mathbb Q}^{n \times n}$.   Then we have the following properties of the right eigenvalues of $A$.

1. If $\lambda \in {\mathbb Q}$ is a right eigenvalue of $A$, then any $\mu \in [\lambda]$ is also a right eigenvalue of $A$ with the same right eigenvectors.

2.  If $\lambda \in {\mathbb C}$ is a right eigenvalue of $A$, then any $\bar \lambda$ is also a right eigenvalue of $A$ with the same right eigenvectors.

3. $A$ has exactly $n$ standard right eigenvalues.

4. $A$ is normal if and only if there is a unitary matrix $U \in {\mathbb Q}^{n \times n}$ such that
\begin{equation} \label{ee4}
U^*AU = D = {\rm diag}(\lambda_1, \cdots, \lambda_n),
\end{equation}
where $\lambda_1, \cdots, \lambda_n$ are right eigenvalues of $A$.

5. $A$ is Hermitian if and only if $A$ is normal, and all the right eigenvalues of $A$ are real.

6. If $A$ is Hermitian, then $A$ has exactly $n$ real right eigenvalues, and there is a unitary matrix $U \in {\mathbb Q}^{n \times n}$ such that (\ref{ee4}) holds with $\lambda_1 \ge \cdots \ge \lambda_n$ as the real right eigenvalues of $A$; $A$ is positive semi-definite if and only if $\lambda_n \ge 0$; $A$ is positive definite if and only if $\lambda_n > 0$.

7. $A$ is unitary if and only if $A$ is normal, and all the right eigenvalues of $A$ are in $S^3$.
\end{Thm}

The following proposition is also not difficult to prove.

\begin{Prop} \label{p2.2}
Suppose that $A \in {\mathbb Q}^{n \times n}$ is Hermitian.   Then $A$ has $n$ real right eigenvalues
$\lambda_1 \ge \cdots, \lambda_n$.   If $\lambda_i \not = \lambda_j$, $\vx$ and $\vy$ are right eigenvectors of $A$, associated with $\lambda_i$ and $\lambda_j$ respectively, then $\vx$ and $\vy$ are orthogonal.  If $\lambda$ is an right eigenvalue of $A$ with multiplicity $k$, then there are $k$ right eigenvectors of $A$, associated with $\lambda$, such that they are orthogonal to each other.
\end{Prop}

\section{Possible Extension of Hua's Theorems to Quaternion Matrices}

In Theorem \ref{t1.1}, a complex skew-symmetric matrix $Z$ is turned to the block diagonal form $\Sigma$ via a complex unitary matrix $U$ such that $UZU^\top = \Sigma$ \cite{Hu44, SW60, Yo61}.    For a quaternion skew-symmetric matrix $Z$, we wish to find a quaternion unitary matrix $U$ such that $UZU^* = \Sigma$, where $\Sigma$ is a block diagonal matrix.  This is called a canonical form under congruence \cite{Ro14}. It is useful in deriving spectral theorem of dual quaternion Hermitian matrices.   In Section 7, we will briefly describe this relation.   However, we only found canonical form under congruence for quaternion symmetric matrices and quaternion skew-Hermitian matrices \cite{Ro14, SW60}.  Thus, we wonder if we can extend Theorem \ref{t1.1} to quaternion skew-symmetric matrices.

Theorem \ref{t1.1} of this paper, i.e., Theorem 7 of \cite{Hu44} is based upon Theorem 5 of \cite{Hu44}.   The first half of Theorem 6 of \cite{Hu44} is easy to be extended to quaternion matrices.

\begin{Prop} \label{p3.1}
Suppose that $Z \in {\mathbb Q}^{n \times n}$ is a skew-symmetric matrix, i.e., $Z = -Z^\top$.   Then
$W \equiv -Z{\bar Z}$ is a positive semidefinite Hermitian matrix.
\end{Prop}
\begin{proof}
Clearly,
$$-Z\bar Z = ZZ^*.$$
Thus, $W \equiv -Z{\bar Z}$ is Hermitian and positive semidefinite.
\end{proof}

The second half of Theorem 6 of \cite{Hu44} says that the characteristic polynomial of $-Z\bar Z$ is a perfect square.   This implies that in the complex case, each distinct positive eigenvalue of $-Z\bar Z$ has even multiplicity.    In the next section we will show that this is no longer true for quaternion matrices.

\section{Quaternion Skew-Symmetric Matrices}

We first consider the case that $n=2$.

\begin{Prop} \label{p4.1}
Suppose that $Z = (z_{ij}) \in {\mathbb Q}^{2 \times 2}$ is a skew-symmetric matrix, and $Z \not = O$.  Then $W = -Z\bar Z$ has a double positive right eigenvalue $|a|^2$, where $a = z_{12} \not = 0$.
\end{Prop}
\begin{proof}
A $2 \times 2$ quaternion skew-symmetric matrix $Z$ has the form
$$Z = \begin{bmatrix} 0 & a \\ -a & 0 \end{bmatrix},$$
where $a \in {\mathbb Q}$.    Since $Z \not = 0$, $a \not = 0$.

Then
$$W \equiv -Z\bar Z = ZZ^* = -\begin{bmatrix} 0 & a \\ -a & 0 \end{bmatrix}\begin{bmatrix} 0 & \bar a \\ -\bar a & 0 \end{bmatrix} = \begin{bmatrix} a\bar a & 0 \\ 0 & a\bar a \end{bmatrix} = |a|^2I_2.$$
Since $Z \not = O$, $a \not = 0$.   Then, $W$ has a double positive right eigenvalue $|a|^2$.
\end{proof}

Hence, if $n=2$, the claim of Theorem 6 of Hua \cite{Hu44} may still be extended to the quaternion case.    We now consider the case that $n = 3$.

If $Z \in {\mathbb Q}^{3 \times 3}$ is a nonzero skew-symmetric matrix, then $Z$ has the form
\begin{equation}  \label{e6}
Z = \begin{bmatrix} 0 & a & c \\ -a & 0 & b \\ -c & -b & 0 \end{bmatrix},
\end{equation}
where $a, b, c \in {\mathbb Q}$, and $a$, $b$ and $c$ are not all $0$.   We may divide the situation to two cases, $a = 0$ and $a \not = 0$.    If $a=0$, then
\begin{equation}  \label{e7}
Z = \begin{bmatrix} 0 & 0 & c \\ 0 & 0 & b \\ -c & -b & 0 \end{bmatrix},
\end{equation}
and not both $b$ and $c$ are $0$.   If $a \not = 0$, without loss of generality, we may assume that $a = 1$, and
\begin{equation}  \label{e8}
Z = \begin{bmatrix} 0 & 1 & c \\ -1 & 0 & b \\ -c & -b & 0 \end{bmatrix}.
\end{equation}
where $b, c \in {\mathbb Q}$.

\begin{Thm} \label{t4.2}
Suppose that $Z \in {\mathbb Q}^{3 \times 3}$, $Z \not = O$, and is in either case (\ref{e7}) or case (\ref{e8}).   Then $W = -Z\bar Z = ZZ^*$ has a zero right eigenvalue and a double positive right eigenvalue if in either case (\ref{e7}) or case (\ref{e8}) with $\bar b\bar c = \bar c\bar b$.    In case (\ref{e8}) with $\bar b\bar c \not = \bar c\bar b$, $W$ has three positive right eigenvalues.   In this case, it is possible that these three positive right eigenvalues are distinct to each other.
\end{Thm}
\begin{proof}  In case (\ref{e7}), since $Z \not = O$, $|b|^2+|c|^2 \not = 0$.   Then we may derive that $W$ has a zero right eigenvalue and a double positive right eigenvalue $|b|^2 + |c|^2$.

We now consider case (\ref{e8}).
We have
$$W = \begin{bmatrix} 1+|c|^2 & c\bar b & -{\bar b}\\
b\bar c & 1 + |b|^2 & \bar c \\ -b & c & |c|^2 + |b|^2 \end{bmatrix}$$
By \cite{WLZZ18}, we have
$$\overline{c\bar b} = b\bar c.$$
This confirms that $W$ is Hermitian.

If $\bar b\bar c = \bar c\bar b$, then by computation, we may derive that $W$ has a zero right eigenvalue and a double positive right eigenvalue $1 + |b|^2 + |c|^2$.

Suppose now that $\bar b\bar c \not = \bar c\bar b$.
Then, we have
$$\bar Z = \begin{bmatrix} 0 & 1 & \bar c \\ -1 & 0 & \bar b \\ -\bar c & -\bar b & 0 \end{bmatrix},$$
Suppose that $\bar Z\vx = \0$. Let $\vx = (x_1, x_2, x_3)^\top$.  Then we have
\begin{equation} \label{eee1}
x_2 + \bar cx_3 = 0,
\end{equation}
\begin{equation} \label{eee2}
-x_1 + \bar bx_3 = 0,
\end{equation}
\begin{equation} \label{eee3}
-\bar cx_1 - \bar bx_2 = 0.
\end{equation}
Substituting (\ref{eee1}) and (\ref{eee2}) to (\ref{eee3}), we have
$$(-\bar c\bar b+\bar b\bar c)x_3 = 0.$$
Since $\bar b\bar c \not = \bar c\bar b$, $-\bar c\bar b+\bar b\bar c \not = 0$.   Thus, $x_3 = 0$.   By (\ref{eee1}) and (\ref{eee2}), we have
$x_1 = x_2 = 0$.   Hence, $\vx = \0$.  Since $Z^\top = -Z$, we have $-Z\bar Z = ZZ^*$.   By Proposition \ref{p3.1}, $-Z\bar Z = ZZ^*$ is Hermitian.   For any $\vx \in {\mathbb Q}^3$, $\vx \not = \0$, by the above discussion, $Z^*\vx = -\bar Z\vx \not = \0$.  Then
$$\vx^*(-Z\bar Z)\vx = \vx^*(ZZ^*)\vx = (Z^*\vx)^*(Z^*\vx) = |Z^*\vx|^2 > 0.$$
Hence, $-Z\bar Z = ZZ^*$ is positive definite.  By Theorem \ref{t2.1}, it has three positive right eigenvalues.

We may let $b = \ii + \jj$ and $c = \ii + 2\jj$.    Then $bc = -3 +\kk \not = -3 -\kk = cb$.
Then
$$W = \begin{bmatrix} 6 & 3+\kk & \ii+\jj\\
3-\kk & 3 & -\ii-2\jj \\ -\ii-\jj & \ii+2\jj & 7 \end{bmatrix}.$$
By computation, we see that $W$ has three distinct right eigenvalues $\lambda_1 \approx 0.0635$, $\lambda_2 \approx 7.5726$ and $\lambda_3 \approx 8.6789$.
\end{proof}

From this theorem, we may derive the following conclusion.

\begin{Thm}\label{t4.3}
Suppose that $Z \in {\mathbb Q}^{3 \times 3}$, $Z \not = O$, and is in the general form (\ref{e6}).   Then
$W \equiv -Z\bar Z = ZZ^*$ has three positive right eigenvalues if and only if $a \not = 0$ and
\begin{equation} \label{e12}
ca^{-1}b \not = ba^{-1}c.
\end{equation}
Otherwise, $W$ has a zero right eigenvalue and a double positive right eigenvalue $|a|^2+|b|^2+|c|^2$.
\end{Thm}
\begin{proof}
By Theorem \ref{t4.2}, $W$ has three positive right eigenvalues if and only if $a \not = 0$ and
\begin{equation} \label{e13}
\overline{a^{-1}b} \cdot \overline{a^{-1}c} \not = \overline{a^{-1}c} \cdot \overline{a^{-1}b}.
\end{equation}
However,
$$\overline{a^{-1}b} \cdot \overline{a^{-1}c} = (a^{-1}b)^*(a^{-1}c)^* = \left((a^{-1}c)(a^{-1}b)\right)^* =
(a^{-1}ca^{-1}b)^*,$$
$$\overline{a^{-1}c} \cdot \overline{a^{-1}b} = (a^{-1}c)^*(a^{-1}b)^* = \left((a^{-1}b)(a^{-1}c)\right)^* =
(a^{-1}ba^{-1}c)^*.$$
Thus, (\ref{e12}) and (\ref{e13}) are equivalent.
\end{proof}

Hence, for $n=3$, in case (\ref{e6}) with (\ref{e12}), $W$ has three positive right eigenvalues.    In this case, the claim of Theorem 6 of Hua \cite{Hu44} cannot be extended to the quaternion case, i.e., Theorem \ref{t1.1} cannot be extended to quaternion matrices directly.    Some changes are needed.

\section{Inverses of Skew-Symmetric Matrices}

Suppose that $Z \in {\mathbb C}^{n \times n}$ is a nonsingular complex skew-symmetric matrix.  Then
$$\left(Z^{-1}\right)^\top = \left(Z^\top\right)^{-1}=(-Z)^{-1}= -Z^{-1},$$
i.e., its inverse is also a skew-symmetric matrix.

Now we consider quaternion skew-symmetric matrices.

A $2 \times 2$ nonzero quaternion skew-symmetric matrix $Z$ has the form
$$Z = \begin{bmatrix} 0 & a \\ -a & 0 \end{bmatrix},$$
where $a \in {\mathbb Q}$, $a \not = 0$.   Then we have
$$Z^{-1} = \begin{bmatrix} 0 & -a^{-1} \\ a^{-1} & 0 \end{bmatrix},$$
which is a skew-symmetric matrix.   In general, we say a $n \times n$ nonzero quaternion skew-symmetric matrix $Z$ is solid if $W \equiv -Z\bar Z = ZZ^*$ is positive definite.    Then we have the following proposition.

\begin{Prop}
For a $3 \times 3$ quaternion skew-symmetric matrix $Z$, if it is invertible, than its inverse is not a skew-symmetric matrix.
\end{Prop}
\begin{proof}  By Theorem \ref{t4.2}, clearly, if $Z$ is invertible, then it must be solid.   We may consider the case (\ref{e8}) with $\bar b \bar c \not = \bar c \bar b$.   The general case may be derived here.   If $Z^{-1}$ is skew-symmetric, then we may assume
$$Z^{-1} = \begin{bmatrix} 0 & d & f \\ -d & 0 & e \\ -f & -d & 0 \end{bmatrix}.$$
By $ZZ^{-1} = I$, we may derive that $d=e=f=0$.   This leads to a contradiction.   Thus, $Z^{-1}$, if exists, must not be skew-symmetric in this case.
\end{proof}

{\bf Question} For $n=3$, can we show that the quaternion skew-symmetric matrix is invertible if it is solid?   Can we give a general form of $Z^{-1}$ in this case?

\section{Basic Quaternion Skew-Symmetric Matrices}

Let $Z \in {\mathbb Q}^{n \times n}$ be a skew-symmetric matrix, $Z \not = O$.   We may define basic quaternion skew-symmetric matrices by induction.   We call all $2 \times 2$ nonzero quaternion skew-symmetric matrices basic quaternion skew-symmetric matrices as the starting point.   For $n > 2$, we say a quaternion skew-symmetric matrix $Z$ is basic if there is no unitary matrix $U$ such that $UZU^* = \Sigma$ where $\Sigma$ is a block diagonal matrix and each block of $\Sigma$ is a lower dimensional skew-symmetric matrix.     Then, for $n=3$, a nonsingular skew-symmetric matrix is a basic quaternion skew-symmetric matrix.

We randomly generate $4 \times 4$ quaternion skew-symmetric matrices and make computation.   We find that
there are examples that there is a $4 \times 4$ quaternion skew-symmetric matrix $Z$ such that $W \equiv -Z\bar Z = ZZ^*$ has four distinct positive right eigenvalues.    This indicates that there are $4 \times 4$ basic quaternion skew-symmetric matrices.   The following is such an example.

\begin{example} Let $Z = Z_1 +Z_2 \ii+ Z_3 \jj + Z_3 \kk$ with
$Z_1 = \left(
         \begin{array}{cccc}
           0 & 1 & 3 & -25 \\
           -1 & 0 & -13 & -10 \\
           -3 & 13 & 0 & 10 \\
           25 & 10 & -10 & 0 \\
         \end{array}
       \right)$, $Z_2 = \left(
         \begin{array}{cccc}
           0 & 3 & 1 & 7 \\
           -3 & 0 & 1 & -6 \\
           -1 & -1 & 0 & 13 \\
           -7 & 6 & -13 & 0 \\
         \end{array}
       \right)$, $Z_3 = \left(
         \begin{array}{cccc}
           0 & 4 & -1 & -3 \\
           -4 & 0 & 1 & 0 \\
           1 & -1 & 0 & 3 \\
           3 & 0 & -3 & 0 \\
         \end{array}
       \right)$, and $Z_4 = \left(
         \begin{array}{cccc}
           0 & -1 & 0 & 9 \\
           1 & 0 & -12 & -3 \\
           0& 12 & 0 & 3 \\
          -9 & 3 & -3& 0 \\
         \end{array}
       \right)$.
Since $Z_i$'s $(i=1,\cdots, 4)$ are skew-symmetric, we have that $Z\in {\mathbb{Q}}^{4\times 4}$ is skew-symmetric. By computation, we find that all the four right eigenvalues of $W=-Z\bar{Z}$ are $\lambda_1 \approx 131.4$, $\lambda_2 \approx 235.5$, $\lambda_3 \approx 1238.3$, and $\lambda_4 \approx 1482.9$.
\end{example}

%{\bf Example}

{\bf Conjecture}  For each $n \ge 4$,  there are $n \times n$ basic quaternion skew-symmetric matrices.

The right answer to the above conjecture will be important for establishing unitary equivalence theorem of quaternion skew-symmetric matrices.

\section{Spectral Theory of Dual Quaternion Matrices}

Quaternions were introduced by Hamilton in 1843 \cite{Ha43}.   In 1873, Clifford \cite{Cl73} introduced dual numbers, dual complex numbers and dual quaternions.   This results a new branch of algebra - geometric algebra or Clifford algebra.    Now, dual numbers, dual complex numbers and dual quaternions have found wide applications in automatic differentiation, geometry, mechanics, rigid body motions, robotics and computer graphics \cite{BK20, BLH19, CKJC16, Da99, Gu11, MKO14, WYL12}.

The further study and applications of dual numbers, dual complex numbers and dual quaternions inevitably lead to the study on dual number matrices, dual complex matrices, dual quaternion matrices and their spectral theories \cite{Br20, Gu21, QL21}.   In particular, recently, Gutin \cite{Gu21} studied spectral theory and singular value decomposition of dual number matrices, and we \cite{QL21} studied  spectral theory and singular value decomposition of dual complex matrices.   Surely, the next step will be the study of the spectral theory and singular value decomposition of dual quaternion matrices.   However, as we know, the development of spectral theory and singular value decomposition of dual quaternion matrices is highly related with the spectral theorem of quaternion skew-symmetric matrices.   This leads us to study this topic. In the following, we briefly describe this relation.

We may denote the set of dual quaternions as $\mathbb {DQ}$.   A dual quaternion $q \in \mathbb {DQ}$ has the form
$$q = q_{st} + q_\I\epsilon,$$
where $q_{st}, q_\I \in \mathbb {DQ}$ are the standard part and the infinitesimal part of $q$ respectively, $\epsilon$ is the infinitesimal unit, satisfying $\epsilon^2 = 0$.  The conjugate of $q$ is
$$\bar q = \bar q_{st} - q_\I\epsilon.$$

Denote the collections of $m \times n$ dual quaternion  matrices by ${\mathbb {DQ}}^{m \times n}$.  Then $A \in {\mathbb {DQ}}^{m \times n}$ can be written as
$$A = A_{st} + A_\I \epsilon,$$
where $A_{st}, A_\I \in {\mathbb {DQ}}^{m \times n}$ are the standard part and the infinitesimal part of $A$ respectively.   In the study of spectral theory of dual quaternion matrices, the core part would be the spectral theory of dual quaternion Hermitian matrices.    Then, $A$ is a dual quaternion Hermitian matrix if and only if $A_{st}$ is a quaternion Hermitian matrix, and $A_\I$ is a quaternion skew-symmetric matrix.   This is the relation between the spectral theory of dual quaternion matrices and the spectral theorem of quaternion skew-symmetric matrices, i.e., a canonical form of quaternion skew-symmetric matrices under congruence.

Then, the answer to the conjecture in the last section will also be important for establishing spectral theory of dual quaternion matrices.

\bigskip

%{\bf Acknowledgment}
%{\bf Data availability statement}    The datasets generated during and/or analysed during the current study are available from the corresponding author on reasonable request.

% \vspace{100pt}

\end{document}